\documentclass{amsart}
\author{Tom\'as Ibarluc\'ia}
\author{Julien Melleray}
\address{Universit\'e de Lyon\\
  Universit\'e Claude Bernard -- Lyon 1 \\
  Institut Camille Jordan, CNRS UMR 5208 \\
  43 boulevard du 11 novembre 1918 \\
  69622 Villeurbanne Cedex \\
  France}
\thanks{Research partially supported by Agence Nationale de la Recherche projects GruPoLoCo (ANR-11-JS01-0008) and ValCoMo (ANR-13-BS01-0006).}

\usepackage{amssymb}
\usepackage[initials,shortalphabetic]{amsrefs}
\usepackage[all]{xy}
\usepackage{mathpazo}
\usepackage{hyperref}
\usepackage{my-macros}
\usepackage{graphicx}
\usepackage{color} 
\usepackage{hyperref}
\numberwithin{equation}{section}
\newcommand{\Clop}{\text{Clop}}
\newcommand{\Prob}{\text{Prob}}

\title{Dynamical simplices and minimal homeomorphisms}

\begin{document}
\begin{abstract}
We give a characterization of sets $K$ of probability measures on a Cantor space $X$ with the property that there exists a minimal homeomorphism $g$ of $X$ such that the set of $g$-invariant probability measures on $X$ coincides with $K$. This extends theorems of Akin (corresponding to the case when $K$ is a singleton) and Dahl (when $K$ is finite-dimensional). Our argument is elementary and different from both Akin's and Dahl's.
\end{abstract}
\maketitle


\section{Introduction}
The study of minimal homeomorphisms (those for which all orbits are dense) on a Cantor space is a suprisingly rich and active domain of research. In a foundational series of papers (see \cite{Herman1992}, \cite{Giordano1995} and \cite{Giordano1999}), Giordano, Herman, Putnam and Skau have pursued the analysis of minimal actions of $\Z$ (and later $\Z^d$), and developed a deep theory. In particular, it is proved in \cite{Giordano1995} that the partition of a Cantor space $X$ induced by the orbits of a minimal homeomorphism $g$ is completely determined, up to a homeomorphism of $X$, by the collection of all $g$-invariant measures. 

Gaining a better understanding of sets of invariant measures then becomes a natural concern, and that is our object of study here: given a Cantor space $X$, and a simplex $K$ of probability measures, when does there exist a minimal homeomorphism $g$ of $X$ such that $K$ is exactly the simplex of all $g$-invariant measures? Downarowicz \cite{Downarowicz1991} proved that any abstract Choquet simplex can be realized in this way; here we are not given $K$ as an abstract simplex, but already as a simplex of measures, so the problem has a different flavour. 

A theorem of Glasner--Weiss \cite{Glasner1995a} imposes a necessary condition: if $g$ is a minimal homeomorphism, $K$ is the simplex of all $g$-invariant measures, and $A,B$ are clopen subsets of $X$ such that $\mu(A) < \mu(B)$ for all $\mu \in K$, then there exists a clopen subset $C \subseteq B$ such that $\mu(C)=\mu(A)$ for all $\mu \in K$. This is already a strong, nontrivial assumption when $K$ is a singleton; in that case the Glasner--Weiss condition is essentially sufficient, as was proved by Akin.

\begin{theorem*}[Akin \cite{Akin2005}]
Assume that $\mu$ is a probability measure on a Cantor space $X$ which is atomless, has full support, and is \emph{good}, that is, for any clopen sets $A,B$, if $\mu(A) < \mu(B)$ then there exists a clopen $C \subseteq B$ such that $\mu(C)=\mu(A)$.

Then there exists a minimal homeomorphism $g$ of $X$ such that the unique $g$-invariant measure is $\mu$.
\end{theorem*}

Following Akin, we say that a simplex is \emph{good} if it satisfies the necessary condition established by Glasner and Weiss. Akin's theorem suggests that, modulo some simple additional necessary conditions, any good simplex of measures could be the simplex of all invariant measures for some minimal homeomorphism. This idea is further reinforced by an unpublished result of Dahl, which generalizes Akin's theorem.

\begin{theorem*}[Dahl \cite{Dahl2008a}]
Let $K$ be a Choquet simplex made up of atomless probability measures with full support on a Cantor space $X$. Assume that $K$ is good and has finitely many extreme points, which are mutually singular. Then there exists a minimal homeomorphism whose set of invariant probability measures coincides with $K$.
\end{theorem*}

Dahl actually obtains a more general result. 
To formulate it, we recall her notation: given a simplex $K$ of probability measures on a Cantor space $X$, let $\text{Aff}(K)$ denote the set of all continuous affine functions on $K$, and $G(K) \subseteq \text{Aff}(K)$
be the set of all functions $\mu \mapsto \int_X f d\mu$, where $f$ belongs to $C(X,\Z)$.

\begin{theorem*}[Dahl \cite{Dahl2008a}]
Let $K$ be a Choquet simplex made up of atomless probability measures with full support on~$X$. Assume $K$ is good and the extreme points of $K$ are mutually singular. If $G(K)$ is dense in $\text{Aff}(K)$, then there exists a minimal homeomorphism whose set of invariant measures is exactly $K$.
\end{theorem*}

As pointed out in the third section of \cite{Dahl2008a}, it follows from Theorem 4.4 in \cite{Effros1981} that $G(K)$ being dense in $\text{Aff}(K)$ is necessary for $G(K)$ to be a so-called simple dimension group, which is in turn necessary for the existence of $g$ as above. The other conditions are also necessary, so Dahl could have formulated her theorem as an equivalence.

Using Lyapunov's theorem, Dahl proves that any finite-dimensional Choquet simplex of probability measures on $X$ with mutually singular extreme points is such that $G(K)$ is uniformly dense in $\text{Aff}(K)$, thus deducing the theorem we stated previously from the one we just quoted. Dahl's proof of her second theorem above uses some high-powered machinery following the Giordano--Herman--Putnam--Skau approach to topological dynamics via dimension groups, $K$-theory, Bratteli diagrams and Bratteli--Vershik maps. By contrast, Akin's proof is elementary and rather explicit, though somewhat long.

Here, we take an approach which is different from both Akin's and Dahl's: we build a minimal homeomorphism preserving a prescribed set of probability measures by constructing inductively a sequence of partitions which will turn out to be Kakutani--Rokhlin partitions for that homeomorphism (we recall the definition of Kakutani--Rokhlin partitions and other basic notions of topological dynamics in the next section). While pursuing this approach, we unearthed a necessary condition for $K$ to be the simplex of invariant measures for some minimal homeomorphism. This led us to the following definition.

\begin{defn*}
We say that a nonempty set $K$ of probability measures on a Cantor space $X$ is a \emph{dynamical  simplex} if it satisfies the following conditions:
\begin{itemize}
\item $K$ is compact and convex.
\item All elements of $K$ are atomless and have full support.
\item $K$ is good.
\item $K$ is \emph{approximately divisible}, i.e.\ for any clopen $A$, any integer $n$ and any $\varepsilon >0$, there exists a clopen $B \subseteq A$ such that $n \mu(B) \in [\mu(A)-\varepsilon,\mu(A)]$ for all $\mu \in K$.
\end{itemize}
\end{defn*}

Note that, assuming that $K$ is good, the assumption that $B \subseteq A$ in the last item above is redundant; we nevertheless include it because this is how approximate divisibility is used in our arguments.

We borrow the terminology ``dynamical simplex'' from Dahl, but our definition is different. Using Lyapunov's theorem as in \cite{Dahl2008a}, it is easy to see that the condition of approximate divisibility is redundant when $K$ is finite dimensional. The main result of this paper is the following.

\begin{theorem} \label{t:mainresult}
Given a simplex $K$ of probability measures on a Cantor space $X$, there exists a minimal homeomorphism $g$ whose set of invariant measures is $K$ if, and only if, $K$ is a dynamical simplex.
\end{theorem}
Our construction produces a minimal homeomorphism $g$ whose set of invariant measures is $K$ and such that the topological full group $[[g]]$ is dense in the closure of the full group $[g]$ (in the terminology of \cite{Bezuglyi2002}, $g$ is \emph{saturated}). When one starts off by assuming that $K$ is the simplex of $T$-invariant measures for some minimal homeomorphism $T$, the existence of such a homeomorphism follows from a combination of theorems of Giordano--Putnam--Skau and Glasner--Weiss, see \cite{Bezuglyi2002}*{Theorem 1.6}. Here we provide an elementary proof of that fact, which seems interesting on its own.

For finite dimensional simplices, our theorem generalizes Dahl's result, showing that the assumption that extreme points are mutually singular is actually a consequence of her other hypotheses. It would be interesting to gain a better understanding of the relationship between her conditions and ours (see Remark \ref{rem:dahlscondition} at the end of the paper).

We would like to point out that the ideas of our construction are different from both Akin's and Dahl's, and relatively elementary; in particular our argument completely bypasses the use of dimension groups, Bratteli diagrams, etc. It is our hope that such ideas could be used to give elementary dynamical proofs of some other theorems of topological dynamics.\\

\noindent \emph{Acknowledgment.} The authors are grateful to G. Aubrun, I. Farah and B. Weiss for interesting comments and discussions. We would also like to thank an anonymous referee for spotting some inacurracies and making valuable suggestions.

\section{Background and notations}
Throughout, $X$ is a Cantor space; we fix some compatible distance on $X$, and whenever we mention the diameter of a set if will be with respect to this distance. We denote by $\Prob(X)$ the compact space of all probability measures on $X$, endowed with its usual topology (which comes from seeing $\Prob(X)$ as a subset of the dual of $C(X)$, endowed with the weak-$*$ topology).

The group $\Homeo(X)$ of all homeomorphisms of $X$ is a Polish group when endowed with the topology of uniform convergence on $X$. Since $X$ is the Cantor space, one can also describe this topology using Stone duality: a homeomorphism of $X$ corresponds to an automorphism of the boolean algebra of clopen sets of $X$, $\Clop(X)$; identifying $\Homeo(X)$ with automorphisms of this algebra yields that a basis of neighbourhoods of identity is given by sets of the form 
$\{g \in \Homeo(X) \colon \forall A \in \mcA \ g(A)= A\}$, where $\mcA$ runs over all clopen partitions of $X$ (note that by compactness all clopen partitions are finite). It is readily checked that the two topologies we just described coincide on $\Homeo(X)$.

\begin{defn}
Given $g \in \Homeo(X)$, its \emph{topological full group} $[[g]]$ is the group of all homeomorphisms $h$ of $X$ such that there exists a clopen partition $A_1,\ldots,A_n$ and integers $n_i$ with the property that for all $x \in A_i$ one has $h(x)= g^{n_i}(x)$.

The \emph{full group} $[g]$ is the group of all homeomorphisms $h$ such that for all $x$ there exists $n$ satisfying $h(x)=g^n(x)$.
\end{defn}

By definition, the topological full group is countable (there are only countably many clopen sets) and contained in the full group. The reason these groups are relevant to our concerns is the following, which follows easily from Proposition 2.6 of \cite{Glasner1995a}.

\begin{theorem}[Glasner--Weiss]\label{t:glasner}
Let $g$ be a minimal homeomorphism. The closure of $[g]$ in $\Homeo(X)$ consists of all homeomorphisms which preserve all $g$-invariant probability measures on $X$.
\end{theorem}

\begin{proof}
Let $H$ denote the group of all homeomorphisms which preserve each $g$-invariant measure. By definition, $H$ is closed and $[g] \subseteq H$, so that $\overline{[g]} \subseteq H$. Towards proving the converse inclusion, pick $h \in H$ and an open neighborhood $O=\{k \in H \colon \forall A \in \mathcal A \ k(A)=h(A)\}$ of $h$, where $\mathcal A$ is a clopen partition of $X$. For any $A \in \mathcal A$ there exists, by Proposition 2.6 of \cite{Glasner1995a}, some $k_A \in [g]$ such that $k_A(A)= h(A)$. Then, the map $k$ defined by setting $k(x)=k_{A}(x)$ whenever $x \in A$ is a homeomorphism (because $\mathcal A$ is a clopen partition, and $h$ as well as each $k_A$ are homeomorphisms) and belongs to $O \cap [g]$.
\end{proof}

It is not always true that $[[g]]$ is dense in $[g]$; when that happens we say that $g$ is \emph{saturated} (this follows terminology introduced in \cite{Bezuglyi2002}). 

We next recall the definition of a \emph{Kakutani--Rokhlin} partition associated to a minimal homeomorphism.

\begin{defn}
A \emph{Kakutani--Rokhlin partition} $\mcT$ associated to a minimal homeomorphism $g$ is a clopen partition of $X$ of the form $A_0 \sqcup \ldots A_k$, where each $A_i$ is further subdivided in $B_{i,0}, \ldots, B_{i,j_i}$ (possibly $j_i=0$) and for all $i$ and all $r \in \{0,\ldots,j_i-1\}$, $g(B_{i,r}) = B_{i,r+1}$. 

The union of all $B_{i,0}$ is called the \emph{base} of the partition, and the union of all $B_{i,j_i}$ is its \emph{top}. Each $A_i$ is called a \emph{column} of the partition, and the definition ensures that $g$ must map the top of the partition onto its base.

\end{defn}

To obtain such a partition, one can first choose a clopen base $B$; then subdivide it into $B_1,\ldots,B_N$, with $B_i$ made up of all $x \in B$ such that $i=\min \{j >0 \colon g^j(x) \in B\}$; and set $B_{i,j}= g^j(B_i)$ for all $j \in \{0,\ldots,i-1\}$. 

Below we represent a Kakutani--Rokhlin partition; the arrows correspond to the action of the homeomorphism on the partition, which is prescribed on all atoms except those contained in the top (all we know there is that the top is mapped onto the base). The base is colored in blue and the top in red; note that on the picture the base and top do not intersect. They are allowed to, but will not intersect as soon as we take a small enough base.

\begin{figure}[!h]
	\center
							\includegraphics[trim =0mm 0mm 0mm 0mm, clip,height=3cm]{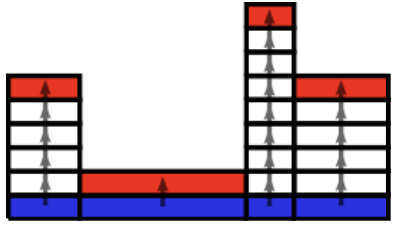}
							\caption{A Kakutani--Rokhlin partition}
\end{figure}

It is a standard, important fact in topological dynamics that, given a minimal homeomorphism $g$, one can produce a sequence of Kakutani--Rokhlin partitions for $g$ whose atoms generate the algebra of clopen sets, and whose top and base have vanishing diameter. Such a sequence naturally defines a basis of neighborhoods of $g$ in $\Homeo(X)$. 

Below, to obtain a minimal homeomorphism with prescribed set of invariant measures, we will define a sequence of partitions which will turn out to be Kakutani--Rokhlin partitions for that homeomorphism.

We recall the definition of a dynamical simplex given in the introduction.

\begin{defn}
We say that a nonempty set $K$ of probability measures on $X$ is a \emph{dynamical  simplex} if it satisfies the following conditions:
\begin{itemize}
\item $K$ is compact and convex.
\item All elements of $K$ are atomless and have full support.
\item $K$ is \emph{good}, i.e.\ for any two clopen sets $A,B$ such that $\mu(A) < \mu(B)$ for all $\mu \in K$, there exists a clopen subset $C \subseteq B$ such that $\mu(C)=\mu(A)$ for all $\mu \in K$.
\item $K$ is \emph{approximately divisible}, i.e.\ for any clopen set $A$, any integer $n$ and any $\varepsilon >0$, there exists a clopen $B \subseteq A$ such that $n \mu(B) \in [\mu(A)-\varepsilon,\mu(A)]$ for all $\mu \in K$.
\end{itemize}
\end{defn}

Given a set $K$ of probability measures, and two clopen sets $A,B$, we use the notation $A \sim_K B$ to denote the fact that $\mu(A)=\mu(B)$ for all $\mu \in K$.
Note that, modulo goodness, approximate divisibility may be stated equivalently by saying that there exist clopen subsets $B_1,\ldots,B_n$ such that $B_1,\ldots,B_n$ are disjoint, contained in $A$, $B_i \sim_K B_j$ for all $i,j$, and $A \setminus \bigcup B_i$ has measure less than $\varepsilon$ for all $\mu \in K$.

We should point out again that we borrow the term ``dynamical simplex'' from Dahl \cite{Dahl2008a}, and that our definition is, at least formally, different from Dahl's : the definition given in \cite{Dahl2008a} includes the assumption that $K$ is a Choquet simplex and extreme points of $K$ are mutually singular, and does not mention approximate divisibility. We briefly discuss the relations between our conditions and Dahl's in Remark \ref{rem:dahlscondition} at the end of the paper.

We note that, when $K$ has finitely many extreme points, the assumption of approximate divisibility is redundant, as follows from the proposition below.

\begin{prop}
Assume that $K$ is a compact subset of $\Prob(X)$, and all the measures in $K$ are atomless and have full support. Then the following properties hold.
\begin{enumerate}
\item For any nonempty clopen set $A$, $\inf \{\mu(A) \colon \mu \in K\} >0$.
\item For any $\varepsilon >0$, there exists $\delta >0$ such that, for any clopen set $A$ of diameter less than $\delta$, one has $\mu(A) \le\varepsilon$ for all $\mu \in K$.
\item If $K$ has finitely many extreme points, then $K$ is approximately divisible.
\end{enumerate}
\end{prop}

\begin{proof}
The first two items are well-known when $K$ is the simplex of all invariant measures for a minimal homeomorphism, and the proofs are simple and similar to that case. We give them for the reader's convenience.

For the first item, assume that there exists a sequence $(\mu_n)$ of elements of $K$ and a nonempty clopen set $A$ such that $\mu_n(A)$ converges to $0$. Then by compactness of $K$ we find some $\mu \in K$ such that $\mu(A)=0$, contradicting the fact that $\mu$ has full support (it is perhaps worth recalling that a sequence $\mu_n$ of elements of $\Prob(X)$ converges to $\mu \in \Prob(X)$ exactly if $\mu_n(A)$ converges to $\mu(A)$ for all clopen set $A$). 

The second item requires a bit more work; we follow the argument of \cite{Bezuglyi-Medynets2008}*{Proposition 2.3} and assume for a contradiction that there exists a sequence of clopen subsets $(A_n)$ of vanishing diameter and $\varepsilon > 0$ such that $\mu_n(A_n) \ge \varepsilon$ for all $n$. Up to passing to a subsequence, we may assume that $(A_n)$ converges to a singleton $\{x\}$ for the Hausdorff distance on compact subsets of $X$, and that $\mu_n$ converges to $\mu \in K$. Let $O$ be any clopen neighborhood of $x$; we will have $A_n \subseteq O$ for all large enough $n$, so that $\mu_n(O) \ge \varepsilon$ for all large $n$. As above, this implies that $\mu(O) \ge \varepsilon$ for all $n$, so that (taking the intersection over all clopen neighborhoods of $x$) $\mu(\{x\})\ge \varepsilon$, contradicting the fact that $\mu$ is atomless.

To see why the third point holds, we argue in a way similar to \cite{Dahl2008a}. Fix $\varepsilon >0$ and let $\mu_1,\ldots,\mu_n$ denote the extreme points of $K$. Lyapunov's theorem on vector measures tells us that 
$$\{(\mu_1(B),\ldots,\mu_n(B)) \colon B \text{ a Borel subset of } A\}$$ is convex. In particular, there exists a Borel subset $B$ of $A$ such that $\mu_i(B)= \frac{1}{n} \mu_i(A)$ for $i \in \{1,\ldots,n\}$. Using the regularity of $\mu_1,\ldots,\mu_n$ we obtain a clopen subset $C \subset A$ such that $\mu_i(C) \in [\frac{\mu_i(A)-\varepsilon}{n}, \frac{\mu_i(A)}{n}]$ for all $i$, which is what we wanted.
\end{proof}

\begin{remark}
If we had assumed that the extreme points of $K$ were mutually singular, we would not have needed Lyapunov's theorem to conclude that $K$ is approximately divisible when $K$ has finitely many extreme points; but we do not need to make this assumption. We also do not include the assumption that $K$ is a Choquet simplex in our definition of a dynamical simplex, as we do not need it in the arguments. Both these assumptions are clearly necessary for $K$ to be the simplex of all invariant measures of an homeomorphism, hence follow from the others, given the main result of the paper.

We do not know if, in general, approximate divisibility is a consequence of the other assumptions (to which one could add the fact that $K$ is a Choquet simplex with mutually singular extreme points, if necessary) as is the case when $K$ is finite-dimensional. The proof above does not seem to adapt: Lyapunov's theorem does extend to more general situations, but this extension (known as Knowles' theorem, see \cite{Diestel1977}*{IX.1.4}) requires the existence of a finite control measure~$\nu$, which will exist only when $K$ has finitely many extreme points. More precisely, one would like to apply the Extension Theorem \cite{Diestel1977}*{I.5.2} to the vector-valued measure $F:\Clop(X)\to C(K)$, $F(A)(\mu)=\mu(A)$, but then, assuming $K$ has countably many mutually singular extreme points, it is not difficult to see that the second item of the theorem fails. Nevertheless, one can certainly prove that approximate divisibility is redundant in some infinite-dimensional situations, for instance when the extreme boundary of $K$ has only one non-isolated point (this was remarked during a conversation with I. Farah).
\end{remark}

\begin{prop}
Assume that $g$ is a minimal homeomorphism of $X$, and that $K$ is the simplex of all $g$-invariant probability measures. Then $K$ is a dynamical simplex.
\end{prop}

\begin{proof}
Clearly $g$-invariant measures always form a compact convex subset of $\Prob(X)$, and when $g$ is minimal any $g$-invariant measure must be atomless and have full support. The fact that the simplex $K$ is good follows from the theorem of Glasner and Weiss recalled in the introduction, so we only need to explain why $K$ is approximately divisible. 

Start from a nonempty clopen set $A$, and consider the map $g_A$ defined by $g_A(x)= g^n(x)$, where $n= \min \{i >0 \colon g^i(x) \in A\}$. Then $g_A$ is a homeomorphism of $A$, and is minimal. The restriction of any $\mu \in K$ defines a $g_A$-invariant measure on $A$, which we still denote by $\mu$.
Pick $N\ge n$ such that $n/N < \varepsilon$. Since $g_A$ is aperiodic, we can find a clopen set $U$ such that $U, g_AU, \ldots, g_A^N U$ are disjoint. In particular, $\mu(U)$ is less than $\mu(A)/N$ for all $\mu \in K$. 

Now, consider a Kakutani--Rokhlin partition of $A$ associated to $g_A$, with base~$U$. Let $C_0,\ldots, C_M$ denote the columns of this partition; we have $C_i= C_{i,0} \sqcup C_{i,1} \ldots \sqcup C_{i,n_i}$, with $n_i \ge N \ge n $. Denote $n_i +1 = k_i n+p$, $p \in \{0,\ldots,n-1\}$. 
For $i \in \{0,\ldots,M\}$ and $j \in \{0,\ldots,n-1\}$, let $B_{i,j}$ denote the union of the levels of $C_i$ of height $<k_i n$ and equal to $j$ modulo $n$, and let $B_j$ be the union of all $B_{i,j}$. Then we have $B_{j} \sim_K B_l$ for all $j,l$, and the complement of their union is of measure less than $n \mu(U) \le \varepsilon \mu(A)$ for all $\mu \in K$.

The following picture is supposed to illustrate the procedure we just described: below, $n=3$; the domains in blue, red and green are $\sim_K$ equivalent, and the measure of the remainder is less than $3$ times the measure of the base, for all $\mu \in K$.
\begin{figure}[!h]
	\center
							\includegraphics[trim =0mm 0mm 0mm 0mm, clip,height=3cm]{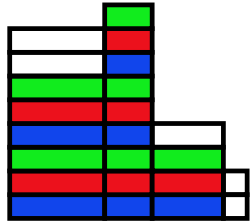}
							\caption{A partition in three $\sim_K$ pieces plus a rest of small measure}
\end{figure}

\end{proof}

The following proposition states an homogeneity property of the algebra $\Clop(X)$ in relation to good simplices.

\begin{prop}\label{prop:homogeneity}
Assume $K$ is a good simplex of probability measures on a Cantor space~$X$ with full support. Let $G= \{g \in \Homeo(X) \colon \forall \mu \in K \ g_* \mu = \mu\}$. If $U,V$ are clopen sets with $U \sim_K V$, then there is $g\in G$ such that $gU=V$.
\end{prop}
\begin{proof}
We construct a $K$-preserving automorphism $g$ of the algebra $\Clop(X)$ by a standard back-and-forth argument. Let $\{A_n\}$, $\{B_n\}$ be two enumerations of $\Clop(X)$, with $A_0=U$ and $B_0=V$. Let $\mcA_0$ be the partition of $X$ into $A_0$ and its complement. We set $gA_0=B_0$, $g(X\setminus A_0)=X\setminus B_0$. Now assume inductively that we have defined $g$ on the atoms of a finite clopen partition $\mcA_n$ such that: (i)~the sets $\{A_i\}_{i<n}$ are contained in the algebra generated by $\mcA_n$, (ii)~the image $g\mcA_n=\{gC:C\in\mcA_n\}$ is a partition of $X$, (iii)~$\mu(gC)=\mu(C)$ for all $C\in\mcA_n$, $\mu\in K$, and (iv)~the sets $\{B_i\}_{i<n}$ are contained in the algebra generated by $g\mcA_n$.

Let $C_1,\dots,C_m$ be the atoms of the partition $\mcA_n$. We take $C^0_i=C_i\cap A_n$, $C^1_i=C_i\setminus A_n$. Since $K$ is good, we can find $D^0_i\subseteq gC_i$ such that $\mu(D^0_i)=\mu(C^0_i)$ for all $\mu\in K$; we set $D^1_i=gC_i\setminus D^0_i$. Now we take $D^{j,0}_i=D^j_i\cap B_n$, $D^{j,1}_i=D^j_i\setminus B_n$. Again, since $K$ is good, we can find a clopen partition $\mcA_{n+1}=\{C^{j,k}_i\}^{j,k<2}_{i<m}$ such that $C^{j,k}_i\subseteq C^j_i$ and $\mu(C^{j,k}_i)=\mu(D^{j,k}_i)$ for each $i,j,k$ and all $\mu\in K$. Then we extend the definition of $g$ to $\mcA_{n+1}$ by setting $gC^{j,k}_i=D^{j,k}_i$. The construction ensures that properties (i)-(iv) are preserved.

At the end we get a $K$-preserving automorphism of $\Clop(X)$ sending $U$ to $V$. By Stone duality, this induces an homeomorphism $g$ as required.\end{proof}


%
%

\section{Construction of a saturated element}

In this section and the next, we fix a dynamical simplex of measures $K$ on a Cantor space $X$, and we let $G$ denote the group of all homeomorphisms $g$ of $X$ such that $g_* \mu= \mu$ for all $\mu \in K$. 

\begin{defn}
We say that $g \in G$ is \emph{$K$-saturated} if for any clopen sets $U,V$ such that $U \sim_K V$ there exists $h \in [[g]]$ with $h(U)=V$.
\end{defn}

\begin{prop}
Assume that $g$ is $K$-saturated. Then $[[g]]$ is dense in $G$ and $g$ is minimal.
\end{prop}

\begin{proof}
It is easy to see that $[[g]]$ being dense in $G$ is equivalent to $g$ being $K$-saturated once once has Theorem \ref{t:glasner} in hand. To see that a $K$-saturated element is minimal, pick any nonempty clopen set $U$. Given any $x \in X$, a sufficiently small clopen neighborhood $V$ of $x$ will be such that $\sup_K \mu(V)  < \inf_K \mu(U)$, thus there exists $g \in G$ such that $gU \supseteq V \ni x$. Hence $X = \bigcup_{g \in G} gU$.

Now, for all $h \in G$ there exists $k \in [[g]]$ such that $kU=gU$, so that $X= \bigcup_{k \in [[g]]} kU$. Since for all $k \in [[g]]$ we have $kU \subseteq \bigcup_{i \in \Z} g^i U$, we obtain that $X= \bigcup_{i \in \Z} g^i U$. This means that $X$ is the unique nonempty open $g$-invariant set, which is the same as saying that $g$ is minimal.
\end{proof}

Next, we introduce partitions which resemble Kakutani--Rokhlin partitions. Essentially, we are trying to build a homeomorphism from a sequence of partitions, rather than the other way around.

\begin{defn}
A \emph{KR-partition} $\mcT$ is a clopen partition of $X$ of the form $A_0 \sqcup \ldots A_k$, where each $A_i$ is further subdivided into $B_{i,0}, \ldots, B_{i,j_i}$ (possibly $j_i=0$) and for all $i$ and all $r,s \in \{0,\ldots,j_i\}$ $B_{i,r} \sim_K B_{i,s}$. 

The union of all $B_{i,0}$ is called the \emph{base} of the partition, and the union of all $B_{i,j_i}$ is its \emph{top}. Each $A_i$ is called a \emph{column} of the partition.

To each KR-partition, one can associate the algebra $\mcA_{\mcT}$ whose atoms are all $B_{i,r}$ with $r< j_i$, and the top of the partition; and the partial automorphism of $\Clop(X)$ with domain $\mcA$ which maps each $B_{i,r}$ to $B_{i,r+1}$ for all $i$ and $r < j_i$, and maps the top of the partition to its base. Note that the ordering of atoms within each column matters.

\end{defn}

We say that a KR-partition $\mcS$ refines a KR-partition $\mcT$ if the base (respectively, top) of $\mcS$ is contained in the base (respectively, top) of $\mcT$, and $\mcS$-towers are obtained by cutting and stacking towers of $\mcT$ on top of each other (see Figure \ref{favorable figure}). More precisely: for each column $C=(D_j)_{0\leq j\leq J}$ of $\mcS$ there exist columns $A_{i_k}=(B_{i_k,j})_{0\leq j\leq j_{i_k}}$ of $\mcT$ ($0\leq k\leq K$) and clopen subsets $S^k_j\subseteq B_{i_k,j}$ such that $J=\sum_{0\leq k\leq K}(j_{i_k}+1)$ and such that, for each $k$, we have $D_{j+\sum_{l<k}(j_{i_l}+1)}=S^k_j$ for every $0\leq j\leq j_{i_k}$.

Note that if $\mcS$ refines $\mcT$ then the algebra and partial automorphism associated to $\mcS$ refine those that are associated to $\mcT$.

\begin{prop}\label{prop:towerdiameter}
Given a KR-partition $\mcT$, and $\varepsilon >0$, there exists a KR-partition $\mcS$ which refines $\mcT$ and which is such that the base and top of $\mcS$ both have diameter less than $\varepsilon$.
\end{prop}

\begin{proof}
We let again $A_0,\ldots,A_k$ denote the columns of $\mcT$, $B_{i,0}$, $B_{i,j_i}$ denote respectively the base and top of the $i$-th column, and $B$ be the base of $\mcT$. 

We begin by describing how to deal with a very favorable particular case, where there exists an integer $n\geq 2$ such that
$\mu(B_{0,0})=\mu(B_{1,0})= \frac{1}{n}\mu(B)$ and both $B_{0,0}$ and $B_{1,j_1}$ have diameter less than $\varepsilon$. Let $R=B \setminus (B_{0,0} \cup B_{1,0})$. Then we have $\mu(R)= (n-2) \mu(B_{0,0})$ for all $\mu \in K$, so, by goodness, as long as $n>2$, we can find a copy $C$ of $B_{0,0}$ inside $R$. Now, the bases of $A_2,\ldots,A_k$ induce a partition of~$C$, $C=\bigsqcup_{j=2}^k (C\cap B_{j,0})$. Furthermore, by goodness (or homogeneity), we can find an equivalent cutting of $B_{0,0}$ into pieces $P_j\sim_K (C\cap B_{j,0})$. We pass on this cutting of $B_{0,0}$ throughout the column $A_0$, and similarly we pass on the cutting of $C$ throughout the columns $A_2,\ldots,A_k$. Finally, we stack each new column starting with $C\cap B_{j,0}$ on the top of the new column based on $P_j$; this gives us a refinement $\mcT'$ of $\mcT$.

As long as $n'=n-1>2$, we repeat this process, only that now we find copies $C_j\subseteq R'=R\setminus C$ of each $P_j$. As before, we cut them with the bases of $A_2,\ldots,A_k$, then imitate this cutting on the corresponding $P_j$ and pass it on throughout the columns; then we stack the new columns based on subsets of $R'$ on top of the corresponding new columns based on subsets of $B_{0,0}$. Once this process has been repeated $n-2$ times, we apply it, lastly, on $A_1$. What we obtain is a refinement $\mcS$ of $\mcT$ whose base is $B_{0,0}$ and whose top is $B_{1,j_1}$.


The picture below illustrates the procedure we just described. The column containing the new base is colored in blue, and the one containing the top is red; we keep track of what happens to them in the picture (in a very simple case for readability).

\begin{figure}[!h]\label{favorable figure}
	\center
							\includegraphics[trim =0mm 0mm 0mm 0mm, clip,height=3cm]{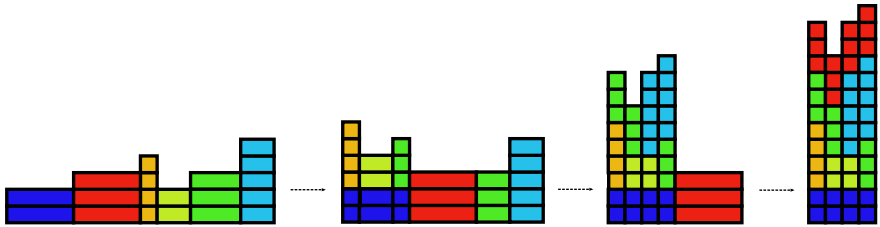}
							\caption{Cutting and stacking in the favorable case}
\end{figure}

That is the last picture that we will include in this article, as the next arguments are a bit harder to illustrate; nevertheless, we invite the reader to draw her own pictures, since we feel that the ideas become more transparent in this way.

To deal with the general case, we use the fact that $K$ is good and approximately divisible to reduce to this favorable case, modulo a small error (which would not appear if $K$ were exactly divisible). First, by cutting $B_{0,0}$ (and throughout $A_0$) if necessary, we can ensure that $B_{0,0}$ has small diameter. Moreover, by picking a subset of $B_{0,j_0}$ of small diameter and cutting again, we can ensure that both the top and base of the first column of our partition have small diameter. Cutting yet again, we make sure that the union $B_{0,j_0}\cup B_{1,j_1}$ has small diameter. Cutting and using goodness once more, we ensure $\mu(B_{0,j_0})= \mu(B_{1,j_1})$ for all $\mu \in K$.

Next, pick some integer $n$ such that $\frac{1}{n} < \mu(B_{0,0}) $ for all $\mu \in K$. Using the fact that $K$ is good and approximately divisible, we find clopen sets $C_0\sim_K C_1\sim_K\ldots\sim_K C_{n-1}$ contained in $B$, pairwise disjoint, such that $C_0 \subseteq B_{0,0}$, $C_1 \subseteq B_{1,0}$ and $E:=B \setminus \bigcup_{i=0}^{n-1} C_i \subseteq B_{0,0}$. We cut $B_{0,0}$ into two pieces, one of which is the error~$E$, inducing a further KR-partition, one column of which has base $E$. We set apart this column $A_E$, that is, we consider $Y=X \setminus A_E$. Then, our current KR-partition of $X$ induces a KR-partition of $Y$; cutting one last time we obtain columns based on $C_0$ and on $C_1$, which we set to be, respectively, the first and the second column of that partition. We have thus obtained a KR-partition (of $Y$) which satisfies the assumptions of the favorable case described above. Applying the stacking procedure given for that case, we obtain a new KR-partition of $Y$ whose base is contained in $B_{0,0}$ and whose top is contained in $B_{1,j_1}$. Finally, considering this partition together with the column $A_E$, we get a KR-partition of $X$ whose base (contained in $B_{0,0}$) and whose top (contained in $B_{0,j_0}\cup B_{1,j_1}$) have both small diameter.
\end{proof}

We say that a partition is \emph{compatible} with a clopen set $U$ if $U$ is a union of atoms of the partition.

\begin{prop}\label{prop:towerrefinement}
Given a KR-partition $\mcT$, and two clopen subsets $U\sim_K V$, one can find a KR-partition $\mcS$ refining $\mcT$, compatible with $U$ and $V$, and such that in each column there are as many atoms contained in $U$ as atoms contained in $V$.
\end{prop}

Note that then, if $g$ is any element of $G$ which extends the partial automorphism associated to $\mcS$, there exists an $h \in [[g]]$ such that $h(U)=V$ (because one can map $U$ to $V$ while only permuting atoms within each column of $\mcS$).

\begin{proof}

First, note that by goodness there is a KR-partition $\mcS$ refining $\mcT$ and compatible with $U,V$: consider the algebra generated by $\mcT$ and $U,V$, then pull back the associated partion of atoms of $\mcT$ 
to the base of $\mcT$ (via an automorphism $g_{i,s}\in G$ mapping $B_{i,s}$ to $B_{i,0}$), and push it back up (using $g_{i,s}^{-1}$). We obtain a new KR-partition~$\mcS$, refining $\mcT$, compatible with $U$ and $V$ (each column has been subdivided into smaller columns, and no stacking has taken place).

Now, for all such KR-partitions, we can associate to any column $C$ the numbers
$$u_C= \sharp\{\text{atoms of } C \text{ contained in } U\},\ v_C = \sharp\{\text{atoms of } C \text{ contained in } V\},$$
and $n_C=u_C-v_C$. Our aim is to find a KR-partition with $u_C=v_C$ for all columns. We distinguish the columns of the following types: $\mcC^+$, the set of columns with $n_C>0$, and $\mcC^-$, those where $n_C<0$. We let $n_\mcS$ denote the maximum value of $|n_C|$ among all columns $C$, and finally let $n$ denote the smallest possible $n_{\mcS}$ among all $\mcS$ refining $\mcT$ and compatible with $U,V$.

We suppose for a contradiction that $n\neq 0$, and we pick $\mcS$ with $n_\mcS=n$. Without loss of generality, we assume that $n=n_C$ for some $C\in\mcC^+$, and we consider the set $\mcD$ of all columns $C$ such that $n_C=n$. Let $B$ be the union of all the bases of columns in $\mcD$, and $B'$ the union of the bases of elements of $\mcC^-$. Then we observe that either $\mu(B)=\mu(B')$ for all $\mu \in K$, or $\mu(B)< \mu(B')$ for all $\mu \in K$ (otherwise $\mu(U)$ and $\mu(V)$ would not be equal). Thus one can build a new KR-partition by cutting and stacking on top of each column of $\mcD$ some element of $\mcC^-$ (just map arbitrarily the union of the tops of elements of $\mcD$ into the union of the bases of $\mcC^-$, then refine accordingly). This has the effect of producing a new KR-partition such that every column in $\mcC^+$ satisfies $n_C<n$. Doing the same (if necessary) with $\mcC^-$, we obtain a contradiction to the minimality of $n$.
\end{proof}

The previous two propositions provide us with the tools to construct the $K$-saturated homeomorphism we were looking for.

\begin{prop}
There exists a $K$-saturated element in $G$.
\end{prop}

\begin{proof}
Fix an enumeration $(U_n,V_n)$ of all pairs of clopen sets $(U,V)$ such that $U \sim_K V$. Using Propositions \ref{prop:towerdiameter} and \ref{prop:towerrefinement}, we build a sequence of $KR$-partitions $\mcS_n$ with the following properties:
\begin{enumerate}
\item For all $n$, $\mcS_{n+1}$ refines $\mcS_n$.
\item For all $n$, $\mcS_n$ is compatible with $U_n, V_n$, and in each column of $\mcS_n$ there are as many atoms contained in $U_n$ as atoms contained in $V_n$.
\item The diameters of the base and top of $\mcS_n$ converge to $0$.
\end{enumerate}

Then, there exists a unique $g \in \Homeo(X)$ which extends the partial automorphisms associated to $\mcS_n$. The construction ensures that $g \in G$, and that $g$ is $K$-saturated.

\end{proof}

\begin{remark}
The set of all $K$-saturated homeomorphisms, as well as the set of all minimal homeomorphisms in $G$, are $G_\delta$ subsets of $G$. The argument above proves that the closure of the set of $K$-saturated elements contains all minimal elements of $G$; thus, in $G$, a generic minimal homeomorphism if $K$-saturated. It would be interesting to determine precisely the closure of the set of all minimal homeomorphisms. It is tempting to believe that it corresponds to the set of all $g \in G$ such that, for any clopen $A$ different from $\emptyset$ and $X$, one has $g(A) \ne A$ (see \cite{Bezuglyi2006}*{Theorem~5.9} for the analogous result in $\Homeo(X)$).
\end{remark}

So far, we have managed to build a $K$-saturated, hence minimal, element which preserves all measures in a given dynamical simplex $K$. It is \emph{a priori} possible that this element preserves measures not belonging to $K$; saturation prevents this from happening. That was the original motivation for trying to build a $K$-saturated element of $G$ rather than merely a minimal homeomorphism belonging to $G$. 
To deal with this issue, we have one remaining task: proving that the set of $G$-invariant probability measures, which by definition contains $K$, in fact coincides with $K$.

\section{Saturated elements cannot preserve unwanted measures}

We will denote the simplex of $G$-invariant probability measures by $K_G$. Given $g\in G$, the simplex of $g$-invariant probability measures will be denoted $K_g$.

\begin{prop}
Let $g$ be a $K$-saturated element. Then $K_g=K_G$.
\end{prop}

\begin{proof}
Clearly, $K_G \subseteq K_g$. For any $\mu \in K_g$ and any $h \in [[g]]$ we have $h_*\mu= \mu$. Since $\{h \colon h_*\mu= \mu\}$ is closed in $\Homeo(X)$, and the closure of $[[g]]$ is $G$ since $g$ is $K$-saturated, we obtain as desired that $h_*\mu = \mu$ for all $h \in G$.
\end{proof}

The last remaining piece of our puzzle is thus the following proposition.

\begin{prop}
We have $K_G\subset K$.
\end{prop}

\begin{proof}
We proceed by contradiction, and assume that $\nu \not \in K$ is such that $g_*\nu=\nu$ for all $g \in G$. By Proposition \ref{prop:homogeneity}, if $U,V$ are clopen sets with $U \sim_K V$, then $\nu(U)=\nu(V)$.

Note first that, if $\mu(A) < \frac{1}{n}$ for all $\mu \in K$, then 
there exist disjoint $A_1,\ldots,A_n$ such that $A_i \sim_K A$ for all $i \in \{1,\ldots,n\}$; since by the remark in the previous paragraph $\nu(A_1)=\ldots=\nu(A_n)=\nu(A)$, we also have $\nu(A) < \frac{1}{n}$. This observation will be used twice below.

Using the Hahn--Banach theorem (see for instance Theorem 3.4(b) of \cite{Rudin1973} for the version we use here, valid in locally convex Hausdorff topological vector spaces; note that we endow the dual of $C(X)$ with the weak* topology, so that its topological dual naturally identifies with $C(X)$, see Theorem 3.10 of \cite{Rudin1973}), we know that there exists a continuous function $f \colon X \to \R$ such that 
$$\forall \mu \in K \  \int_X f d\mu < \int_X f d \nu  .$$
Replacing $f$ by $f + \max\{ |f(x)| \colon x \in X \}$, and using the fact that all our measures are probability measures, we may assume that $f(x)\geq 0$ for all $x$. Using uniform continuity of $f$, we may further assume that $f$ only takes finitely many values, which are all non-negative rational numbers; multiplying by a large enough integer, we finally reduce to the case when $f$ takes finitely many integer values. Hence we have finitely many clopen sets $(A_i)_{1 \le i \le N}$ and positive integers $n_i$ such that
$$\forall \mu \in K \  \sum_{i=1}^N n_i \mu(A_i) < \sum_{i=1}^N n_i \nu(A_i).$$
Pick integers $p,q >1$ such that 
$$ \forall \mu \in K \  \sum_{i=1}^N n_i \mu(A_i) < \frac{p}{q} <\sum_{i=1}^N n_i \nu(A_i) .$$  
Using the fact that $K$ is approximately divisible, we may find for all $i$ some clopen sets $B_{i,1} \sim_K B_{i,2} \sim_K\ldots \sim_K B_{i,p}$ contained in $A_i$ such that 
$\mu(A_i \setminus \bigcup B_{i,j})$ is arbitrarily small for all $\mu \in K$, hence also $\nu(A_i \setminus \bigcup B_{i,j})$ is arbitrarily small.

Thus, $\frac{\nu(A_i)}{p}-\nu(B_{i,1})$ can be made arbitrarily small, so we can ensure that 
$$\forall \mu \in K \  \sum_{i=1}^N n_i \mu(B_{i,1}) < \frac{1}{q} < \sum_{i=1}^N n_i \nu(B_{i,1}) . $$

We can then build a set $B$ which is a disjoint union of $n_1$ copies of $B_{1,1}$ (that is, $n_1$ clopen sets which are $\sim_K$-equivalent to $B_{1,1}$), $n_2$ copies of $B_{2,1}$, etc; we have
$$ \forall \mu \in K \ \mu(B) < \frac{1}{q} \quad \text{ and } \quad \nu(B) > \frac{1}{q}\ .$$
This contradicts the observation made at the beginning of the proof.
\end{proof}

\begin{cor}
Assume that $g$ is a $K$-saturated element of $G$. Then $K=K_g$.
\end{cor}

\begin{proof}
We have $K\subset K_g$ by definition of $G$, and the converse inclusion follows from the previous propositions.
\end{proof}

We have finally proved Theorem \ref{t:mainresult}. 

\begin{remark}\label{rem:dahlscondition} Using the idea of the proof of the previous proposition, one can check that goodness and approximate divisibility imply that any affine function on $K$ of the form $\mu\mapsto\int_X fd\mu$, where $f\in C(X,[0,1])$, can be approximated arbitrarily well by an affine function of the form $\mu\mapsto\int_X\chi_B$, where $\chi_B$ is the characteristic function of a clopen set. This implies, in the terminology of Dahl, that $G(K)$ is dense in $\text{Aff}(K)$. 

It seems clear that $G(K)$ being dense in $\text{Aff}(K)$ and approximate divisibility are related conditions; it appears to follow from Dahl's arguments in \cite{Dahl2008a} that, whenever $K$ is a Choquet simplex and $G(K)$ is dense in $\text{Aff}(K)$, $K$ must be approximately divisible. 
\end{remark}

We conclude by discussing a further line of enquiry suggested by our work here. As mentioned at the end of the introduction, our argument gives an elementary proof of the fact that, given a minimal homeomorphism $g$, there exists a saturated homeomorphism $h$ preserving the same measures as $g$. It follows from a theorem of Krieger \cite{Krieger1980}, the proof of which is relatively short and elementary, that any two saturated homeomorphisms preserving the same measures are orbit equivalent. 
It then becomes interesting to try and find a dynamical proof of the fact that, if $g$ is a minimal homeomorphism, $S$ is a saturated minimal homeomorphism, and $K_g=K_S$, then $g$ and $S$ are orbit-equivalent : combining such a proof, the main result of this paper, and Krieger's theorem, one would obtain a new dynamical proof of the Giordano--Putnam--Skau relating orbit equivalence and invariant measures.

\bibliography{mybiblio}

\begin{bibdiv}
\begin{biblist}

\bib{Akin2005}{article}{
      author={Akin, Ethan},
       title={Good measures on {C}antor space},
        date={2005},
        ISSN={0002-9947},
     journal={Trans. Amer. Math. Soc.},
      volume={357},
      number={7},
       pages={2681\ndash 2722 (electronic)},
      review={\MR{MR2139523 (2006e:37003)}},
}

\bib{Bezuglyi2006}{article}{
      author={Bezuglyi, Sergey},
      author={Dooley, Anthony~H.},
      author={Kwiatkowski, Jan},
       title={Topologies on the group of homeomorphisms of a {C}antor set},
        date={2006},
        ISSN={1230-3429},
     journal={Topol. Methods Nonlinear Anal.},
      volume={27},
      number={2},
       pages={299\ndash 331},
      review={\MR{MR2237457}},
}

\bib{Bezuglyi2002}{article}{
      author={Bezuglyi, Sergey},
      author={Kwiatkowski, Jan},
       title={Topologies on full groups and normalizers of {C}antor minimal
  systems},
        date={2002},
        ISSN={1027-1767},
     journal={Mat. Fiz. Anal. Geom.},
      volume={9},
      number={3},
       pages={455\ndash 464},
      review={\MR{MR1949801 (2003j:37014)}},
}

\bib{Bezuglyi-Medynets2008}{article}{
      author={Bezuglyi, S.},
      author={Medynets, K.},
       title={Full groups, flip conjugacy, and orbit equivalence of {C}antor
  minimal systems},
        date={2008},
        ISSN={0010-1354},
     journal={Colloq. Math.},
      volume={110},
      number={2},
       pages={409\ndash 429},
         url={http://dx.doi.org/10.4064/cm110-2-6},
      review={\MR{2353913 (2009b:37011)}},
}

\bib{Dahl2008a}{incollection}{
      author={Dahl, H.},
       title={Dyanmical {C}hoquet simplices and {C}antor minimal systems},
        date={2008},
   booktitle={Cantor minimal systems and {AF}-equivalence relations},
        note={PhD Thesis, Copenhagen, available at
  http://hdl.handle.net/11250/249732},
}

\bib{Downarowicz1991}{article}{
      author={Downarowicz, Tomasz},
       title={The {C}hoquet simplex of invariant measures for minimal flows},
        date={1991},
        ISSN={0021-2172},
     journal={Israel J. Math.},
      volume={74},
      number={2-3},
       pages={241\ndash 256},
         url={http://dx.doi.org/10.1007/BF02775789},
      review={\MR{1135237 (93e:54029)}},
}

\bib{Diestel1977}{book}{
      author={Diestel, J.},
      author={Uhl, J.~J., Jr.},
       title={Vector measures},
   publisher={American Mathematical Society, Providence, R.I.},
        date={1977},
        note={With a foreword by B. J. Pettis, Mathematical Surveys, No. 15},
      review={\MR{0453964 (56 \#12216)}},
}

\bib{Effros1981}{book}{
      author={Effros, Edward~G.},
       title={Dimensions and {$C^{\ast} $}-algebras},
      series={CBMS Regional Conference Series in Mathematics},
   publisher={Conference Board of the Mathematical Sciences, Washington, D.C.},
        date={1981},
      volume={46},
        ISBN={0-8218-1697-7},
      review={\MR{623762 (84k:46042)}},
}

\bib{Giordano1995}{article}{
      author={Giordano, Thierry},
      author={Putnam, Ian~F.},
      author={Skau, Christian~F.},
       title={Topological orbit equivalence and {$C^*$}-crossed products},
        date={1995},
        ISSN={0075-4102},
     journal={J. Reine Angew. Math.},
      volume={469},
       pages={51\ndash 111},
      review={\MR{1363826 (97g:46085)}},
}

\bib{Giordano1999}{article}{
      author={Giordano, Thierry},
      author={Putnam, Ian~F.},
      author={Skau, Christian~F.},
       title={Full groups of {C}antor minimal systems},
        date={1999},
        ISSN={0021-2172},
     journal={Israel J. Math.},
      volume={111},
       pages={285\ndash 320},
      review={\MR{MR1710743 (2000g:46096)}},
}

\bib{Glasner1995a}{article}{
      author={Glasner, Eli},
      author={Weiss, Benjamin},
       title={Weak orbit equivalence of {C}antor minimal systems},
        date={1995},
        ISSN={0129-167X},
     journal={Internat. J. Math.},
      volume={6},
      number={4},
       pages={559\ndash 579},
      review={\MR{MR1339645 (96g:46054)}},
}

\bib{Herman1992}{article}{
      author={Herman, Richard~H.},
      author={Putnam, Ian~F.},
      author={Skau, Christian~F.},
       title={Ordered {B}ratteli diagrams, dimension groups and topological
  dynamics},
        date={1992},
        ISSN={0129-167X},
     journal={Internat. J. Math.},
      volume={3},
      number={6},
       pages={827\ndash 864},
      review={\MR{MR1194074 (94f:46096)}},
}

\bib{Krieger1980}{article}{
      author={Krieger, Wolfgang},
       title={On a dimension for a class of homeomorphism groups},
        date={1979/80},
        ISSN={0025-5831},
     journal={Math. Ann.},
      volume={252},
      number={2},
       pages={87\ndash 95},
         url={http://dx.doi.org/10.1007/BF01420115},
      review={\MR{593623 (82b:46083)}},
}

\bib{Rudin1973}{book}{
      author={Rudin, Walter},
       title={Functional analysis},
   publisher={McGraw-Hill Book Co., New York-D\"usseldorf-Johannesburg},
        date={1973},
        note={McGraw-Hill Series in Higher Mathematics},
      review={\MR{0365062}},
}

\end{biblist}
\end{bibdiv}

\end{document}